\documentclass{amsart}
\usepackage[utf8]{inputenc}
\usepackage{hyperref}
\usepackage{amsmath, amsfonts, amsthm, amssymb, mathrsfs}
\DeclareMathOperator{\cl}{cl}

\DeclareMathOperator{\ident}{id}
\DeclareMathOperator{\convhull}{co}
\newcommand\F{Fr\'echet}
\newcommand\RKHS{{\mathcal{H}}}
\newtheorem{theorem}{Theorem}[section]
\newtheorem{proposition}[theorem]{Proposition}
\newtheorem{lemma}[theorem]{Lemma}

\newtheorem{definition}[theorem]{Definition}
\newtheorem{notation}[theorem]{Notation}
\newtheorem*{question}{Question}

\theoremstyle{remark}
\newtheorem{remark}[theorem]{Remark}
\usepackage{geometry}
\usepackage{enumitem}

\title{Banach Intermediate Spaces for Gaussian Fr\'{e}chet Spaces}
\author{Yifei Zheng}
\address{School of Mechanical Engineering, Purdue University, West Lafayette, IN, 47907,USA}
\email{zheng432@purdue.edu}
\author{Zachary Selk*}
\address{Department of Mathematics, Purdue University, West Lafayette, IN, 47907, USA, }
\email{\it{Corresponding author}, zselk@purdue.edu}
\date{\today}


\begin{document}
\maketitle

\begin{abstract}
    In this article, we show that every centered Gaussian measure on an infinite dimensional separable \F{} space $X$ over $\mathbb R$ admits some full measure Banach intermediate space between $X$ and its Cameron-Martin space. We provide a way of generating such spaces and, by showing a partial converse, give a characterization of Banach intermediate spaces. Finally, we show an example of constructing an $\alpha$-H\"older intermediate space in the space of continuous functions, $\mathcal C_0[0, 1]$ with the classical Wiener measure.
\end{abstract}

\section{Introduction}
Gaussian measures on topological vector spaces are a central object of study in probability theory. The primary examples of Gaussian measures on topological vector spaces are continuous Gaussian processes realized as Gaussian measures on the space of continuous functions $\mathcal C([0,T],\mathbb R^n)$ under the sup norm --- the canonical example of a continuous Gaussian process being Brownian motion. Since then, Gaussian measures have been generalized to Banach spaces and furthermore to more general topological vector spaces. See \cite{Bogachevbook} for a general introduction to infinite dimensional Gaussian measure theory. 

It is interesting to see what properties of the canonical example of classical Wiener measure generalize to Gaussian measures on more general spaces. In \cite{Baldi}, the author recalled that there is an ``intermediate space" for Brownian motion. Let $W_0^{1,2}([0,T])$ denote the Cameron-Martin or Reproducing Kernel Hilbert Space of absolutely continuous functions with $L^2$ weak derivative. Let $\mathcal C_s^{\alpha}([0,T])$ denote the set of ``small" $\alpha$-H\"{o}lder functions $f$ so that the modulus of continuity
\[\omega(\delta):=\sup_{0\leq s\leq t\leq T, |t-s|\leq \delta} |f(t)-f(s)|\]
is such that
\[\lim_{\delta \to 0^+} \frac{\omega(\delta)}{\delta^\alpha}=0.\]
Denote by $\mathcal C_0([0,T])$ the space of continuous functions starting at $0$, with the sup norm and the classical Wiener measure. The author of \cite{Baldi} recalled that for $0< \alpha <\frac{1}{2}$, there is the sequence of compact embeddings
\[W_0^{1,2}\hookrightarrow \mathcal C_s^{\alpha}\hookrightarrow \mathcal C_0, \]
with $\mu(\mathcal C_s^{\alpha})=1.$ The author showed that this is a general phenomenon for Gaussian measures on separable Banach spaces. More precisely, the author showed the following theorem.  

\begin{theorem}[{\cite[Thm.~1.1]{Baldi}}]\label{t:Baldi_intermediate}
Let $E$ be a separable Banach space and $\mu$ a centered Gaussian measure on $E$ and $\RKHS$ the corresponding RKHS. Then there exists a Banach space $\tilde{E}$, separable and such that $\mu(\tilde{E}) = 1$ and the embeddings $E \hookleftarrow \tilde{E} \hookleftarrow \RKHS$ are compact.
\end{theorem}

In this article, we present a generalization to the case of separable Fr\'echet spaces. We furthermore characterize the full measure intermediate spaces through ``shape functions". More precisely, we have the following theorem.

\begin{theorem}\label{t:main}
Let $\mu$ be a centered Gaussian measure on a separable infinite dimensional Fr\'echet space $X$ with Cameron-Martin space $\RKHS$. Then there exists a linear subspace $E$ with a norm $\|\cdot\|_E$ lower semicontinuous in metric, such that 
\[\RKHS\hookrightarrow E \hookrightarrow X,\]
with the embeddings compact. In particular, there exists $E$ such that $\mu(E) = 1$. Conversely, any space $E$ (full measure or not) with such property can be generated by a ``shape function" $\phi$.
\end{theorem}

See Prop.~\ref{p:full_measure_existence} along with Prop.~\ref{p:converse} for full statements and proofs. One application of Theorem \ref{t:main} is exponential tightness of Gaussian measures, as noted in \cite{Baldi}.

\section{Preliminaries}

In this section, we collect some useful results from functional analysis and from infinite dimensional Gaussian measure theory. 

\subsection{Gaussian measures and Cameron-Martin spaces}

In this subsection, we collect some results on Gaussian measures on topological vector spaces. For a comprehensive introduction, see \cite{Bogachevbook}. 

\begin{notation}\label{d:X}
Let $X$ be an infinite dimensional separable \F{} space over $\mathbb{R}$ equipped with a centered Gaussian measure $\mu$.
\end{notation}

Since $X$ is a separable metric space, $\mu$ is Radon (\cite[17.11]{KechrisBook}).

The following definitions are taken from \cite{Bogachevbook}, which we reproduce here for convenience.

\begin{definition}[Gaussian measure, centered]
Let $\mathcal{E}(X)$ be the minimal $\sigma$-algebra, with respect to which all continuous linear functionals on $X$ are measurable. A probability measure $\mu$ defined on $\mathcal{E}(X)$ generated by $X^*$ is \emph{Gaussian} if, for any $f \in X^*$ the induced measure $\mu \circ f^{-1}$ is a Gaussian measure on $\mathbb{R}$. The measure $\mu$ is \emph{centered} if all measures $\mu \circ f^{-1}$, $f \in X^*$ are centered.
\end{definition}

Without loss of generality, we assume that the support of $\mu$ is $X$, by which $\mu$ is strictly positive on $X$. Separability under this assumption turns out to be equivalent to $\mu$ being Radon (Thm.~\ref{t:Radon_RKHS_separable}, \ref{t:RKHS_dense}).

\begin{definition}[Covariance operator]
Let $(\cdot)^*$ denote the continuous dual and $(\cdot)'$ the algebraic dual. The operator $R_\mu : X^* \to (X^*)'$,
\begin{equation}
R_\mu : f \mapsto \left((\cdot) : g \mapsto \int_X (f(x) - \overline{f}) (g(x) - \overline{g})\,\mu({\rm d}x)\right)
\end{equation}
is called the \emph{covariance operator} of $\mu$, where overline denotes the $\mu$-mean operator, defined as:
\begin{equation}
\overline{f} = \int f(x)\,\mu({\rm d}x).
\end{equation}
\end{definition}
We hereafter without loss of generality only consider centered Gaussian measures, for which $\overline{f} = \overline{g} = 0$, and
\begin{equation}
R_\mu(f)(g) = \int_X f(x) g(x)\,\mu({\rm d}x).
\end{equation}

The $L^2(\mu)$ closure of $X^*$, denoted by $X_\mu^*$, is the reproducing kernel Hilbert space (RKHS) of $(X, \mu)$; for a centered measure, $R_\mu$ can be easily extended to $X_\mu^*$.

\begin{definition}[Cameron-Martin space]\label{d:CM-space}
In light of Riesz representation theorem, the space $\{h \in X : h = R_\mu(g), g \in X_\mu^*\}$ is known as the Cameron-Martin space, denoted by $\RKHS$. The topology on $\RKHS$ is characterized through $h = R_\mu(g)$; that is, $\langle h_1, h_2 \rangle_{X_\mu^*} = \langle g_1, g_2 \rangle_{\RKHS}$.
\end{definition}
This definition is equivalent to $\RKHS = \{h \in X : \sup\{l(h) : l \in X^*, R_\mu(l)(l) \leq 1\} < \infty\}$. $\RKHS$ will be used from now on to denote the Cameron-Martin space of $X$. The primary reason we care about the Cameron-Martin space $\RKHS$ is the celebrated Cameron-Martin theorem, which states that Gaussian measures are quasi-invariant under translation of $x \in X$ if and only if $x \in \RKHS$. We include some theorems used in the proofs later regarding the properties of Cameron-Martin space. The theorems may have been specialized to our setting, including making the notation consistent.

\begin{theorem}[{\cite[3.2.2]{Bogachevbook}}]\label{t:RKHS_invariant}
Let $\gamma$ be a Radon Gaussian measure on a locally convex space $X$ which is continuously and linearly embedded into a locally convex space $Y$. Then the set $H(\gamma)$ is independent of whether $\gamma$ is considered on $X$ or on $Y$.
\end{theorem}

\begin{theorem}[{\cite[3.2.4]{Bogachevbook}}]\label{t:H_ball_compact}
Let $\mu$ be a Radon Gaussian measure on $X$. Then the closed unit ball $U_H$ from $\RKHS$ is compact in $X$.
\end{theorem}

\begin{theorem}[{\cite[3.2.7]{Bogachevbook}}]\label{t:Radon_RKHS_separable}
Let $\mu$ be a Radon Gaussian measure on $X$. Then $\RKHS$ is separable.
\end{theorem}


\begin{theorem}[{\cite[2.5.5, see also 2.5.2]{Bogachevbook}}]\label{t:zero_one}
Let $\mu$ be a Radon Gaussian measure on $X$ and let $L$ be a $\mu$-measurable affine subspace in $X$. Then either $\gamma(L) = 0$ or $\gamma(L) = 1$.
\end{theorem}

\begin{theorem}[{\cite[3.6.1]{Bogachevbook}}]\label{t:RKHS_dense}
Let $\mu$ be a centered Radon Gaussian measure on $X$. Then the topological support of $\mu$ coincides with $\overline{\RKHS}$ (closure in $X$). In particular, $\overline{\RKHS}$ is separable.
\end{theorem}

\subsection{Topological vector spaces}

Before we proceed, we make some remarks on the topology of metric spaces and topological vector spaces in general that we've found useful for the proofs. See \cite{Rudinbook} for more information.

\begin{notation}
Let $X$ be a topological vector space and $U \subset X$. Denote the closure of $U$ in $X$ by $\cl_X U$. Denote the convex hull of $U$ by $\convhull(U)$. Then the closed convex hull of $U$ is $\cl_X \convhull(U)$.
\end{notation}

\begin{definition}[Bounded, topologically]
Let $S \subset X$. $S$ is \emph{bounded} if and only if for every neighborhood $U$ of $0$ there exists some $r \in \mathbb{R}$ such that $S \subset r U$.
\end{definition}



\begin{theorem}[{\cite[3.1.12]{NBbook}}]\label{t:completeness}
Let $\mathcal{T}_s$ and $\mathcal{T}_w$ be Hausdorff group topologies for a group $X$. Let $V_s(0)$ denote the filter of $\mathcal{T}_s$-neighborhoods of $0$. If $\mathcal{T}_s$ is stronger than $\mathcal{T}_w$ and there exists a base $\mathcal{B}_w$ of $\mathcal{T}_w$-complete sets for $V_s(0)$, then $X$ is $\mathcal{T}_s$-complete.
\end{theorem}
A ``neighborhood'' in \cite{NBbook} refers to a set that contains an open set. We do not use this convention; except for this one case, all neighborhood refers to open neighborhoods.

\begin{theorem}[{\cite[3.20bc]{Rudinbook}}]\label{t:convex_hull_preserves_compactness}
If $X$ is a locally convex topological vector space and $E \subset X$ is totally bounded, then $\convhull(E)$ is totally bounded. If $X$ is a \F{} space and $K \subset X$ is compact, then $\cl_X \convhull(K)$ is compact. 
\end{theorem}

\subsection{Objective and some relevant prior work}

\begin{definition}\label{d:intermediate_space}
A linear subspace $E \subset X$ is an \emph{intermediate space} if and only if $\RKHS \hookrightarrow E \hookrightarrow X$ and both embeddings are compact.
\end{definition}

The objective of this paper is to construct intermediate spaces with additional desirable properties, namely being normed, complete, and of full measure. The following two theorems give similar results, which we present here for comparison.

\begin{theorem}[{\cite[3.6.5]{Bogachevbook}}]\label{t:Bogachev_intermediate}
Let $\mu$ be a Radon measure on a \F{} space $X$. Then there exists a linear subspace $E \subset X$ with the following properties:
\begin{enumerate}
\setlength\parskip{0pt}
\item There is a norm $\|\cdot\|_E$ on $E$ with respect to which $E$ is a reflexive separable Banach space such that the closed unit ball in this norm is compact in $X$;
\item $|\mu|(X \setminus E) = 0$.
\end{enumerate}
\end{theorem}
{
\renewcommand{\thetheorem}{\ref{t:Baldi_intermediate}}
\begin{theorem}[{\cite[Thm.~1.1]{Baldi}}]
Let $E$ be a separable Banach space and $\mu$ a centered Gaussian probability on $E$ and $\RKHS$ the corresponding RKHS. Then there exists a Banach space $\tilde{E}$, separable and such that $\mu(\tilde{E}) = 1$ and the embeddings $E \hookleftarrow \tilde{E} \hookleftarrow \RKHS$ are compact.
\end{theorem}
}

The result we will show is an extension of Thm.~\ref{t:Baldi_intermediate} to \F{} spaces. It is also mostly implied by Thm.~\ref{t:Bogachev_intermediate}, for $|\mu|(X \setminus E) = 0$ implies $E$ is full measure. Given $E$ is full measure, the Cameron-Martin space of $X$ is also that of $E$; then $\RKHS$ embeds compactly into $E$ by Thm.~\ref{t:H_ball_compact}. However, the proof we will present defines the intermediate space by manipulating the ``shape'' of the unit ball using Definition \ref{d:phi}, through which it provides a local characterization of Banach intermediate spaces.

\section{Results}

We divide the construction process into three steps (and correspondingly into three subsections). We firstly look for spaces that are ``small'' enough to embed into $X$ compactly; then for spaces that are ``large'' enough into which $\RKHS$ can embed compactly; and finally we investigate some properties and flexibilities of thus constructed spaces. We only use the topological property of $\RKHS$ that its closed unit ball is compact in $X$ for the first two steps. Thus, the construction till then is valid for finding intermediate spaces between two arbitrary spaces where one embeds compactly into another. Later, we construct intermediate spaces with other properties such as being full measure.

\subsection{Embedding into \texorpdfstring{$X$}{X} compactly}
Recall the definition of $X$ in Def.~\ref{d:X} and that a compact operator is one that maps bounded set to relatively compact (totally bounded) sets. For convenience in construction, we only look for those that map closed balls to compact sets.

Consider the closed unit ball of the intermediate space.
\begin{notation}\label{n:K}
Let $K \subset X$ denote a symmetric convex compact set throughout this section.
\end{notation}
A symmetric convex set is balanced and therefore $K$ is absolutely convex. 
\begin{definition}\label{d:E}
Let $E$ be the linear span of $K$; that is, $E = \{ r x : x \in K, r \in \mathbb{R} \} = \bigcup_{r > 0} r K$.
\end{definition}
By construction, $K$ is absorbing in $E$. The Minkowski functional $p_K$ defines a seminorm on $E$. Let $M_K = \sup \{ d(0, x) : x \in K \}$; since $K$ is compact, $M_K$ is finite. Without loss of generality, assume $M_K = 1$. Note that then $p_K(x) \geq d(0, x)$ for all $x$, which shows the $p_K$-topology on $E$ is finer than its subspace topology inherited from $X$. Since $K$ is compact, $K$ is bounded in $X$ ($K$ does not contain any nontrivial vector subspace), therefore $p_K$ separates points, and it is in fact a norm. By Thm.~\ref{t:completeness} ($\mathcal{B}_w = \{r K : r \in \mathbb{R}\}$), $E$ is complete under $p_K$. $E$ is hence a Banach space.

\begin{lemma}\label{l:proper_containment}
Let $A$ and $X$ be F-spaces . If there exists a compact linear map $\Lambda : A \to X$, then $\Lambda$ is not surjective.
\end{lemma}
\begin{proof}
Since compact maps are bounded, $\Lambda$ is continuous. Since $\Lambda$ is compact, there exists some $A$-neighborhood $U$ of $0$ such that $\Lambda(U) \subset K$ where $K$ is compact in $X$. Since both $A$ and $X$ are F-spaces, if $\Lambda$ is surjective, then it is open by open mapping theorem. In particular, $\Lambda(U)$ is open and relatively compact. This contradicts with $X$ being infinite dimensional and not locally compact.
\end{proof}
Let $E$ be an intermediate space. Letting i) $A := E$, $X := X$, ii) $A := \RKHS$, $X := E$, and $\Lambda = \ident$ be the identity map, the lemma shows that $\RKHS \subsetneq E \subsetneq X$.

Since the norm topology is no weaker than metric subspace topology, $K = \cl_X \convhull(S) = \cl_E \convhull(S)$. $K$ is the closed unit ball of $E$. The norm as a function is not continuous in the metric topology because Minkowski functional is continuous if and only if the underlying set is a neighborhood of $0$. However, $\|\cdot\| : E \to \mathbb{R}$ is lower semicontinuous under metric because the norm closed balls $a K$ are closed in metric as well for all $a \in \mathbb{R}$. Reader may refer to \cite[Ch.\ IV \S 6.2]{BourbakiBook} for more information about semicontinuity.

\subsection{Embedding \texorpdfstring{$\RKHS$}{H} compactly}
\begin{notation}
Define $\|x\| : x \mapsto p_K(x)$ for all $x \in E$.
\end{notation}

\begin{notation}
Unless otherwise specified, the topology on $X$, $E$, $\RKHS$ are the metric, norm, and inner product topology, respectively. Let $B^\RKHS$, $B^K$, $B^d$ be the open unit ball of $\RKHS$, $E$, and $X$ (centered at 0), respectively. Let $S^\RKHS$, $S^K$, $S^d$ be the unit sphere of corresponding spaces. The notation $x + B^\cdot_r$ denotes a ball centered at $x$ with radius $r$.
\end{notation}
In particular, $B^\cdot_r$ is identical to $r B^\cdot$ for normed spaces, but not for general metrizable spaces. For our convenience, we fix a complete and translation-invariant metric $d$ such that $\sup_{x \in B^\RKHS} d(0, x) = 1$. A ball in the following proofs always has positive radius and never degenerates to a singleton or empty set.

We now define ``shape functions'' which is used to ensure that $\RKHS$ is embedded compactly into $E$.

\begin{definition}[Shape function]\label{d:phi}
A shape function is any function $\phi : S^\RKHS \to \cl_X(B^d) \cap \RKHS \setminus B^\RKHS$, such that when its domain and codomain are considered under the metric topology of $X$:
\begin{enumerate}\setlength{\parskip}{0pt}
\item[a)] for all $x$, there exists some $k(x)\in \mathbb R_{>0}$ so that $\phi(x) = k(x) x$ and $\phi(-x) = -\phi(x)$, \emph{Denote $|\phi|(x) = k$};
\item[b)] there exists some compact $T \subset X$ such that for every neighborhood $U \subset X$ of $T$, $\phi^{-1}\{U\} \supset V_U \cap S^\RKHS$ where $V_U$ is a metric neighborhood of $0$;
\item[c)] $\phi(S^\RKHS \setminus B^d_\epsilon) \subset \RKHS$ is (inner product) bounded for all $\epsilon > 0$ (cf.\ d);
\item[d)] $\lim_{x \to 0} |\phi|(x) = \infty$.
\end{enumerate}
\end{definition}
An example of shape functions is $\phi : x \mapsto \lfloor d(0, x)^\alpha \rfloor x$ for any $\alpha \in (-1, 0)$ ($T = \{0\}$ for b).
\begin{notation}
The symbol $\phi$ will denote a shape function from now on. Upon clarification, $\phi$ may also denote its homogeneous extension to $\RKHS$, that is, $\phi(x) = \sqrt{\langle x, x \rangle}\phi(\sqrt{\langle x, x \rangle}^{-1} x)$ for all $x \neq 0$ and $\phi(0) = 0$.
\end{notation}
\begin{definition}[Generated space]\label{d:construction}
Let $S = \phi(S^\RKHS)$ and $K = \cl_X \convhull (S)$. Let $\tilde{E}$ be the span of $K$ (cf. Def.~\ref{d:E}), the closed linear span of $\RKHS$ therein is called the \emph{generated space}, denoted by $E$.
\end{definition}
Specifically, by ``generated space", we mean the space $E$ equipped with the induced norm $p_K$. We will verify this agrees with Notation \ref{n:K} in the following proof.

\begin{proposition}\label{t:existence_of_K}
The $K$ in Def.~\ref{d:construction} agrees with Notation~\ref{n:K} and its generated $E$ is an intermediate space.
\end{proposition}
\begin{proof}
We claim $S$ (the range of $\phi$) is totally bounded in $X$ by showing that it has a finite cover consisting of metric $r$-balls for all $r$.
Since $T$ is compact, the cover $\{ t + B^d_r \}_{t \in T}$ has a finite subcover, the union $U$ over which is a (finite) union of open sets and therefore open. By Def.~\ref{d:phi}b, $\phi^{-1}\{U\}$ contains some metric ball $B^d_{\epsilon_r}$. Then by Def.~\ref{d:phi}c, $\phi(S^\RKHS \setminus B^d_{\epsilon_r}) \subset B^\RKHS_R$ for some $R$, where the latter is totally bounded in $X$ by Thm.~\ref{t:H_ball_compact}. Hence $S \subset U \cup B^\RKHS_R$; $S$ then has a finite cover consisting of metric balls with radius $r$ for all $r > 0$. By Thm.~\ref{t:convex_hull_preserves_compactness} and property a), $K$ is compact and symmetric, respectively; Def.~\ref{d:construction} agrees with Notation \ref{n:K}. We see the norm unit ball $B^K$ is totally bounded in $X$.

($\RKHS \hookrightarrow E$ compactly.) We show that $B^\RKHS$ is relatively compact in $E$. Consider sequential compactness. Since $B^\RKHS$ is relatively compact in $X$ (Thm.~\ref{t:H_ball_compact}), every sequence taking values therein has a convergent (therefore Cauchy) subsequence $\{x_i\}_{i \in \mathbb{N}}$. We show $\{x_i\}$ is norm-Cauchy.
By Def.~\ref{d:phi}d, $\liminf_{N \to \infty} \{ |\phi|(x_i - x_j) : i, j > N\} = \infty$. Then there exists some unbounded $\chi : \mathbb{N} \to \mathbb{R}$ such that $\chi(N) (x_i - x_j) \in K$.
By homogeneity of norm, $\limsup_{N \to \infty} \{ \|x_i - x_j\| : i, j > N \} \leq \lim_{N \to \infty} 2 \chi(N)^{-1} = 0$; the factor $2$ is due to $B^\RKHS - B^\RKHS = 2 B^\RKHS$ and the homogeneous extension of $\phi$. Hence the sequence is norm-Cauchy.
\end{proof}

\subsection{Constructing intermediate spaces to specifications}
\begin{remark}[Denseness]
We note that $\RKHS$ may not be dense in thus constructed $E$, for although $K = \cl_X \convhull S$ where $S \in \RKHS$, it may happen that $\cl_E \convhull S \subsetneq K$ as the norm topology is finer. Since $\mu(E) = 1$, by Thm.~\ref{t:RKHS_invariant}, $\RKHS$ is also the Cameron-Martin space of $E$. Then by Thm.~\ref{t:RKHS_dense} (and definition of topological support), $\mu(\cl_E H) = 1$. This shows that every full measure intermediate space admits a subspace with full measure where $\RKHS$ is dense in norm.
\end{remark}

$K$ can also be chosen such that $E$ is a full measure subspace, by the following results:
\begin{lemma}\label{l:enlargement}
Given any intermediate space with closed unit ball $K$, any (absolutely convex, compact) $K' \supset K$ also generates an intermediate space.
\end{lemma}
\begin{proof}
Since $K' \supset K$, $E_{K'} \supset E_K$. In particular, on $E_K$, the topology induced by $K$ is no coarser than the subspace topology inherited from $E_{K'}$. Hence from $E_K$ to $E_{K'}$, we see the topology as being coarsened or unchanged. In either case, compactness is preserved. 
\end{proof}
Note Lemma~\ref{l:enlargement} does not imply $E_{K'}$ is distinct from $E_K$.
\begin{proposition}\label{p:full_measure_existence}
Every infinite dimensional separable \F{} space $X$ admits an intermediate space with full measure.
\end{proposition}
\begin{proof}
By Thm.~\ref{t:RKHS_dense}, $\mu$ is strictly positive (supported everywhere) and $\mu(B^d) > 0$. Since $\mu$ is Radon, there exists some $K' \subset B_d$ such that $\mu(K') > 0$. Define a new $K$ to be the closed convex hull of symmetric hull of $K' \cup K$ and generate a new $E$. Since $K \subset E$, $\mu(E) \geq \mu(K) > 0$. By Thm.~\ref{t:zero_one}, $\mu(E) = 1$.
\end{proof}

It remains unclear to the authors if explicit requirements can be put onto shape functions such that $\mu(K)$ or $\mu(E)$ is positive for general $X$.

\begin{remark}[Reflexivity]
After applying Prop.~\ref{p:full_measure_existence}, the same approach used for proof of Thm.~\ref{t:Bogachev_intermediate} (citing \cite[Ch.~5~\S4~Thm.~1]{DiestelNotes}) is still valid if a reflexive space is desirable, although apparently $\phi$ loses control over the unit ball. We note that $E$ (as in Def.~\ref{d:construction}) is weakly compactly generated (by the unit ball of $\RKHS$) and \cite{DiestelNotes} contains many properties and characterizations of such spaces.
\end{remark}

\begin{proposition}
If $\RKHS$ is dense in an intermediate space $E$, then $E$ is separable.
\end{proposition}
\begin{proof}
Since $\mu$ is a Radon Gaussian measure, $\RKHS$ is separable by Thm.~\ref{t:Radon_RKHS_separable}. Since $\RKHS$ is dense in $E$, for all $r > 0$, $\RKHS + B^K_r$ covers $E$. Since $\RKHS$ is separable, there exists some countable dense subset $S \subset \RKHS$ such that $S + B^\RKHS_r$ covers $\RKHS$. Since $S + B^\RKHS_r \subset S + B^K_r$, the latter also covers $\RKHS$. We see that $(S + B^K_r) + B^K_r$ covers $E$ for all $r > 0$. Since $B^K$ is convex \cite[p.~38]{Rudinbook}, we have that $B^K_r + B^K_r = B^K_{2 r}$. Given that the initial choice of $r$ is arbitrary, we conclude that $S$ is a countable dense subset of $E$ as well.
\end{proof}

Finally, we show a converse to Prop.~\ref{t:existence_of_K}.
\begin{proposition}\label{p:converse}
Given any (complete) normed intermediate space $E$ where the norm is lower semicontinuous with respect to the topology on $X$, there exists a shape function $\phi$ that generates it as described in Def.~\ref{d:construction}. (Completeness is redundant by Thm.~\ref{t:completeness}.)
\end{proposition}
\begin{proof}
Define $|\phi| : f \mapsto \|f\|^{-1}$ and $\phi : f \mapsto |\phi|(f) f$ (then homogeneously extended). By slight abuse of notation, we will continue to use $B^K$ to denote the open unit ball of the given $E$.

Def.~\ref{d:phi}a is satisfied by construction. Choose $T = \cl_X B^K$, which is compact since $E \hookrightarrow X$ compactly. Def.~\ref{d:phi}b is then satisfied because $T$ contains the image of $\phi$. Def.~\ref{d:phi}c is true because $|\phi|(f) = \|f\|^{-1} \leq d(0, f)^{-1}$ up to a constant ($B^K$ is totally bounded in $X$); the latter meets the requirement.

(Def.~\ref{d:phi}d.) Let $\{x_i\}_{i \in \mathbb{N}} \in S^\RKHS$ be a sequence that converges to 0 in metric. We claim that it also converges to 0 in norm. To see this, assume not. Then there exists some $r$ such that $x_i \notin B^K_r$ for infinitely many $i$; such $x_i$ form a sequence $\{x_i\}_{i \in I}$ ($I \subset \mathbb{N}$). Since $S^\RKHS$ is relatively compact in $E$, $\{x_i\}_{i \in I} \in S^\RKHS$ has a convergent subsequence in $E$. But since both the metric topology and norm topology are Hausdorff, $\{x_i\}_{i \in \mathbb{N}}$ and $\{x_i\}_{i \in I}$ shall have the same limit. This shows that every sequence taking values in $S^\RKHS$ converges to $0$ in norm if (and only if) in metric. By homogeneity of norm, $\lim_{x \to 0} |\phi|(x) \geq \lim_{i \to \infty} \|x_i\|^{-1} = \lim_{x \to 0} \|x\|^{-1} = \infty$.


Since $\cl_E B^K = \cl_X B^K$ is itself convex, the norm closed unit ball coincides with its closed convex hull in $X$. Therefore, the generated space of $\phi$ is $E$.
\end{proof}

\section{Example: H\"older spaces and the Wiener space}

In this section, we demonstrate the classical case of an intermediate space that was first noted in \cite{Baldi}.

Recall that  $\mathcal C_0([0,T],\mathbb R)$ is the space of all real valued continuous functions thereon whose initial values are 0 (i.e.\ $f(0) = 0$ for all $f \in \mathcal C_0$). Let $\mathcal C_0$ have the $\sup$-norm, which is equivalent to $\sup_{a, b \in [0, 1]} |f(a) - f(b)|$. Let $W_0^{1, 2}$ denote the space of all absolutely continuous real valued functions whose weak derivative is square integrable and initial value is 0. Let $\mathcal C_0^{0, \alpha}$ denote space of $\alpha$-H\"older functions with initial value 0. It is well known that $W_0^{1, 2} \hookrightarrow \mathcal C_0^{0, \alpha} \hookrightarrow \mathcal C_0$ compactly for $\alpha \in (0, \frac{1}{2})$. We show that there is a corresponding $\phi$ that generates a subspace of the H\"older spaces $\mathcal C_0^{0, \alpha}$ satisfying the four properties required in Def.~\ref{d:phi} without resorting to Proposition \ref{p:converse}.

Denote $\|f\|_\alpha = \sup_{a, b \in [0, 1], a \neq b} \frac{f(a) - f(b)}{|a - b|^\alpha}$, which is a (separating) norm on $C_0^{0, \alpha}$. Define $|\phi| : f \mapsto \|f\|_\alpha^{-1}$ ($0 < \alpha < \frac{1}{2}$) to be the reciprocal of H\"older constant and $\phi : f \mapsto |\phi|(f) f$. (Both defined on $S^{W_0^{1, 2}}$.) By construction, $\phi$ satisfies Def.~\ref{d:phi}a. Def.~\ref{d:phi}b is satisfied by simply choosing $T$ to be the closed unit ball of $C_0^{0, \alpha}$. We note that for small H\"older space $C_s^\alpha$, we may choose $T = \{0\}$. The other two properties are given below.

\begin{proposition}[Property c]
If $\sup |f| > \epsilon > 0$, then there exists some $M$ such that $|\phi|(f) < M$.
\end{proposition}
\begin{proof}
\begin{equation}
|\phi|(f)^{-1} = \sup_{a, b \in [0, 1], a \neq b} \frac{f(a) - f(b)}{|a - b|^\alpha} \geq \frac{\sup |f|}{1^\alpha} = \sup |f| > \epsilon.
\end{equation}
Choose $M = \epsilon^{-1}$.
\end{proof}

\begin{proposition}[Property d]
For any $M > 0$, there exists some $\epsilon > 0$ such that for all $f \in W_0^{1, 2}$, if $\int f'(x)^2\,{\rm d}x = 1$ and $\sup |f| < \epsilon$, then $|\phi|(f) > M$.
\end{proposition}
\begin{proof}
By Cauchy-Schwartz inequality and the fundamental theorem of calculus, for any $a, b$,
\begin{equation}
f(a) - f(b) = \int_a^b 1 \cdot f'(x)\,{\rm d}x \leq \sqrt{\int_a^b 1^2\,{\rm d}x} \sqrt{\int_a^b f'(x)^2\,{\rm d}x} \leq \sqrt{b - a}.
\end{equation}
Hence for any $a, b$, $|f(a) - f(b)| \leq \min\{\sqrt{b - a}, 2 \epsilon\} \leq (2 \epsilon)^{1 - 2 \alpha} |a - b|^\alpha$ (the latter is obtained by solving $M (2 \epsilon)^{2 \alpha} = 2 \epsilon$ for $M$). Since $\alpha \in (0, \frac{1}{2})$, $\lim_{\epsilon \to 0} (2 \epsilon)^{1 - 2 \alpha} = 0$. Therefore, appropriate $\epsilon$ can always be found to make the H\"older constant sufficiently small and thereby $|\phi|$ sufficiently large.
\end{proof}

Now we carry out the procedure of Def.~\ref{d:construction} and find $K$. $K = \cl_{\mathcal{C}_0} \convhull(\phi(S^{W_0^{1, 2}}))$ and $\phi(S^{W_0^{1, 2}}) = W_0^{1, 2} \cap B^{\mathcal{C}_0^{0, \alpha}}$; clearly $K \subseteq \cl_{\mathcal{C}_0} B^{\mathcal{C}_0^{0, \alpha}}$. Since the H\"older norm is lower semicontinuous in $\sup$-norm, $\cl_{\mathcal{C}_0^{0, \alpha}} B^{\mathcal{C}_0^{0, \alpha}} = \cl_{\mathcal{C}_0} B^{\mathcal{C}_0^{0, \alpha}}$; this shows \underline{$K$ is contained in the $\alpha$-H\"older closed unit ball}. Since $W_0^{1, 2}$ is a linear space and $B^{\mathcal{C}_0^{0, \alpha}}$ is convex, $K = \cl_{\mathcal{C}_0} (W_0^{1, 2} \cap B^{\mathcal{C}_0^{0, \alpha}})$. Consider the closure of $W_0^{1, 2}$ in $\alpha$-H\"older space, that is, the small $\alpha$-H\"older space $\mathcal{C}_s^\alpha$. Since $W_0^{1, 2}$ is dense in $C_s^\alpha$ and $B^{\mathcal{C}_0^{0, \alpha}}$ being the open unit ball has nonempty interior, $K \supseteq \cl_{\mathcal{C}_0^{0, \alpha}} (W_0^{1, 2} \cap B^{\mathcal{C}_0^{0, \alpha}}) \supseteq C_s^\alpha \cap B^{\mathcal{C}_0^{0, \alpha}}$. That is, \underline{$K$ contains the open unit ball of $C_s^\alpha$}. We conclude that $E$ is set-theoretically bounded between $C_0^{0, \alpha}$ and $C_s^\alpha$, with a norm equivalent to the $\alpha$-H\"older norm. In particular, $\mu(E) \geq \mu(C_s^\alpha) = 1$.

\section{Conclusions}
In this article, we showed that any centered Gaussian measure on a separable Fr\'echet space has a full measure Banach intermediate space. We conclude with a question. 
\begin{question}
Can $\phi$ be chosen so that the generated space has full measure for non-normable spaces, without appealing to inner regularity? 
\end{question}
\begin{question}
Is there a Banach intermediate space $E$ of some \F{} space $X$ where the norm is not lower semicontinuous in $X$, or equivalently, where the closed unit ball of $E$ is not closed in $X$?
\end{question}

\bibliographystyle{plain}
\bibliography{bibliography}

\end{document}